\documentclass[12pt,centertags]{amsart}
\usepackage{amsmath,amstext,amsthm,a4,amssymb,amscd}
\usepackage[mathscr]{eucal}
\usepackage{mathrsfs}
\usepackage{epsf}

\usepackage{amsmath,amstext,amsthm,a4,amssymb,amscd,mathrsfs}
\usepackage{mathtools}
\usepackage[mathscr]{eucal}
\usepackage{mathrsfs}
\usepackage{epsf}
\numberwithin{equation}{section}

\newtheorem{thm}{Theorem}[section]

\newtheorem{prop}[thm]{Proposition}

\theoremstyle{definition}

\theoremstyle{definition}

\theoremstyle{definition}
\newtheorem{defn}[thm]{Definition}

\newcommand{\be}{\begin{eqnarray}}
\newcommand{\ee}{\end{eqnarray}}

\newcommand{\comment}[1]{}

\begin{document}

\title{A remark on $\Lambda^2$-enlargeable manifolds}
 
\author{Guangxiang Su}

\address{Chern Institute of Mathematics \& LPMC, Nankai
University, Tianjin 300071, P.R. China}
\email{guangxiangsu@nankai.edu.cn}

\begin{abstract}
In this note, we consider the case where the condition ``{constant} near infinity" in the definition of $\Lambda^2$-enlargeable manifolds is replaced by the condition ``{locally constant} near infinity" and prove that a $\Lambda^2$-enlargeable manifold in this modified sense still cannot carry a complete Riemannian metric of positive scalar curvature. As a consequence, we give another proof of Wang-Zhang's theorem on the generalized Geroch conjecture for complete spin manifolds.
\end{abstract}

\maketitle

\section{Introduction}

The $\Lambda^2$-enlargeable manifold was introduced by Gromov-Lawson (\cite{GL83}). A famous theorem of Gromov-Lawson (\cite{GL83}) states that a $\Lambda^2$-enlargeable manifold cannot carry a complete Riemannian metric of positive scalar curvature.

Let $W$ be a closed $\Lambda^2$-enlargeable manifold and $M$ be a noncompact connected spin manifold without boundary with ${\rm dim}M={\rm dim}W$. Wang and Zhang (\cite{WZ}) proved that the connected sum $M\#W$ cannot carry a complete Riemannian metric of positive scalar curvature using the result in \cite{Z20}. If $M$ is a closed spin manifold, then $M\# W$ is a closed $\Lambda^2$-enlargeable manifold (\cite{GL83}). So it is natural 
to ask whether $M\#W$ is a $\Lambda^2$-enlargeable manifold for the case that $M$ is a noncompact spin manifold. For a noncompact $M$, from the constructions in \cite{SZ,WZ} (which goes back to \cite{GL83}), one finds that the condition that the maps from the covering manifolds to the standard sphere are constant near infinity is not satisfied for $M\#W$. In fact, the maps are only locally constant near infinity. In this note, we consider the $\Lambda^2$-enlargeable manifold in this case and prove that the $\Lambda^2$-enlargeable manifold in the current sense also cannot carry a complete Riemannian metric of positive scalar curvature.

\begin{defn}[\cite{GL83}]
A $C^1$-map $\varphi:X\to Y$ between Riemannian manifolds is said to be $(\epsilon,\Lambda^2)$-contracting, if for all $x\in X$, the map $\varphi_*:\Lambda^2(T_x X)\to \Lambda^2(T_{\varphi(x)}Y)$ satisfies
$$|\varphi_*(V_x\wedge W_x)|\leq \epsilon |V_x\wedge W_x|$$
for any $V_x,W_x\in T_xX$.
\end{defn}

In the following definition of $\Lambda^2$-enlargeable Riemannian metrics, we replace ``constant near infinity" in \cite[Definition 7.1]{GL83} by ``locally constant near infinity". In \cite[Definition 1.10]{S}, Shi also considered this condition, but from a different motivation, also, there is no overlap between Shi's main results and the results proved in the present paper. 

\begin{defn}\label{df1.2}
Let $M$ be a connected manifold without boundary. A Riemannian metric on $M$ is called $\Lambda^2$-enlargeable if given any $\epsilon>0$, there exist a covering manifold $M_\epsilon\to M$ such that $M_\epsilon$ is spin and a smooth map $f_\epsilon:M_\epsilon\to S^{{\rm dim}M}(1)$ which is $(\epsilon,\Lambda^2)$-contracting with respect to the lifted metric, {\bf locally constant} near infinity and of non-zero degree. A connected (not necessarily compact) manifold is said to be $\Lambda^2$-enlargeable if all the Riemannian metrics (not necessarily complete) on $M$ are $\Lambda^2$-enlargeable. 
\end{defn}

\begin{prop}\label{l1.4}
Let $W$ be a closed $\Lambda^2$-enlargeable manifold and $M$ be a noncompact connected spin manifold without boundary with ${\rm dim}M={\rm dim}W$, then $M\# W$ is a $\Lambda^2$-enlargeable manifold in the sense of Definition \ref{df1.2}.
\end{prop}

From Proposition \ref{l1.4} one finds that there are manifolds which are $\Lambda^2$-enlargeable in the sense of Definition \ref{df1.2} but not $\Lambda^2$-enlargeable in the sense of \cite[Definition 7.1]{GL83}. For example, since the torus $T^n$ is a closed $\Lambda^2$-enlargeable manifold, then $M\# T^n$ is a $\Lambda^2$-enlargeable manifold in the current sense for a noncompact spin manifold $M$.

The main theorem of this paper can be stated as follows, which extends \cite[Theorem 6.12]{GL83} to the current case.

\begin{thm}\label{th1.3}
A manifold $M$ without boundary which is $\Lambda^2$-enlargeable in the sense of Definition \ref{df1.2}, cannot carry a complete Riemannian metric of positive scalar curvature. 
\end{thm}

By Proposition \ref{l1.4} and Theorem \ref{th1.3}, we give another proof of \cite[Theorem 1.1]{WZ}, which states that $M\# W$ cannot carry a complete Riemannian metric of positive scalar curvature.

The rest of this paper is organized as follows. In Section 2, we will give a proof of Proposition \ref{l1.4}. In Section 3, we will give a proof of Theorem \ref{th1.3}.

\section{Proof of Proposition \ref{l1.4}}

Let $g^{TW}_0$ be a fixed Riemannian metric on $TW$. We fix a point $p\in W$. For any $r\geq 0$, let $B^W_{p}(r)=\{y\in W: d(p,y)\leq r\}$. As in \cite{SZ,WZ}, let $b_0$ be a fixed sufficiently small number. Then the connected sum $M\# W$ can be constructed so that the hypersurface $\partial B^W_p(b_0)$, which is the boundary of $B^W_p(b_0)$, cuts $M\# W$ into two parts: one part $W\setminus B^W_p(b_0)$ and the other part coming from $M$ (by attaching the boundary of a ball in $M$ to $\partial B^W_p(b_0)$). Let $g_0^{T(W\setminus B^W_p(b_0))}$ be the restricted metric on $T(W\setminus B^W_p(b_0))$ from $g^{TW}_0$.

Let $g^{T(M\# W)}$ be an arbitrary Riemannian metric on $T(M\# W)$. From the metric $g^{T(M\# W)}$, by restriction we get a metric $g^{T(W\setminus B^W_p(b_0))}$ on $T(W\setminus B^W_p(b_0))$. For  any $\epsilon>0$, let $\pi: \widehat{W}_\epsilon\to W$ be a covering manifold satisfying \cite[Definition 7.1]{GL83}, carrying the lifted geometric data of $(W, g^{TW}_0)$. Especially there exist a smooth map $f:\widehat{W}_\epsilon\to S^{{\rm dim}M}(1)$ and a compact subset $K_\epsilon\subset \widehat{W}_\epsilon$ such that $f$ is constant on $\widehat{W}_\epsilon\setminus K_\epsilon$. As in \cite{SZ, WZ}, the connected sum $M\# W$ is lifted naturally to $\widehat{W}_\epsilon$ and we denote the resulting manifold by $\widehat{M}\# \widehat{W}_\epsilon$. We lift the metric $g^{T(M\# W)}$ to $\widehat{M}\# \widehat{W}_\epsilon$ and extend the map $f$ to $\widehat{M}\# \widehat{W}_\epsilon$ by setting that $f|_{\widehat{M}\#(\widehat{W}_\epsilon\setminus K_\epsilon)}$ is constant. Since $\overline{W\setminus B^W_p(b_0)}$ is compact, the metrics $g^{T(W\setminus B^W_p(b_0))}_0$ and $g^{T(W\setminus B^W_p(b_0))}$ are equivalent. Then by the construction in \cite{WZ}, we have a map $\widehat{f}:\widehat{M}\#\widehat{W}_\epsilon\to S^{{\rm dim}M}(1)$ which is $(c\epsilon,\Lambda^2)$-contracting with respect to the lifted metric of $g^{T(M\# W)}$ for some constant $c>0$, locally constant near infinity and of non-zero degree. Then the metric $g^{T(M\# W)}$ is a $\Lambda^2$-enlargeable metric in the sense of Definition \ref{df1.2} and by definition $M\# W$ is a $\Lambda^2$-enlargeable manifold in the sense of Definition \ref{df1.2}.

\section{Proof of Theorem \ref{th1.3}}
In this section we give a proof of Theorem \ref{th1.3} using the methods in \cite{SWZ} and \cite{LSWZ}.

Let $g^{TM}$ be a complete Riemannian metric on $TM$ and $k^{TM}$ be the associated scalar curvature. We argue by contradiction. Assume that 
$$k^{TM}>0\ \ {\rm over}\ \ M.$$

Following the proof of \cite[Theorem 6.12]{GL83}, we consider another metric on $TM$ defined by $k^{TM}g^{TM}$. By definition, for the metric $k^{TM}g^{TM}$ and any $\epsilon>0$, there exists a covering 
$$\pi_\epsilon: M_\epsilon\to M$$
such that $M_\epsilon$ is spin and there exists a smooth map
$$f_\epsilon: M_\epsilon\to S^{{\rm dim}M}(1)$$
which is $(\epsilon,\Lambda^2)$-contracting for the lifted metric of $k^{TM}g^{TM}$, {\bf locally constant} outside a compact subset $K_\epsilon$ and of non-zero degree.

Let $g^{TM_\epsilon}=\pi^*_\epsilon g^{TM}$ be the lifted metric of $g^{TM}$ and $k^{TM_\epsilon}=\pi^*(k^{TM})$. Set ${\rm dim}M=n$.

\subsection{Construct a closed manifold}
In this subsection, we briefly recall from \cite{SWZ} and \cite{LSWZ} the construction of a closed manifold from $(M_\epsilon, g^{TM_\epsilon})$.

Following \cite[Theorem 1.17]{GL83}, we choose a fixed point $x_0\in M_\epsilon$ and let $d:M_\epsilon\to \mathbb{R}^+$ be a regularization of the distance function ${\rm dist}(x,x_0)$ such that 
{
  \begin{equation}
    \label{eq:fun-d}
  |\nabla d|(x)\leq \frac{3}{2},
\end{equation}
for any $x\in M_\epsilon$.}
Set
\begin{equation}\label{728}
  B_{\epsilon,m}=\{x\in M_\epsilon: d(x)\leq m\},\ m\in \mathbb{N}.
\end{equation}

Since $K_\epsilon$ is compact, we can choose a sufficiently large $m$ such that $K_\epsilon\subseteq B_{\epsilon,m}$. This implies
\begin{equation}
  \label{eq:df-b}
  {\rm Supp}({\rm d} f_\epsilon) \subseteq K_\epsilon \subseteq B_{\epsilon,m}.
\end{equation}

Following \cite{GL83}, we take a compact hypersurface $H_{\epsilon,3m}\subseteq M_\epsilon\setminus K_\epsilon$, cutting $M_\epsilon$ into two parts such that the compact part, denoted by $M_{H_{\epsilon,3m}}$, contains $B_{\epsilon,3m}$. Then $M_{H_{\epsilon,3m}}$ is a compact smooth manifold with boundary $H_{\epsilon,3m}$. Note that the number of connected components of $M_\epsilon\setminus M_{H_{\epsilon,3m}}$ is finite. Let $\{Y_j\}_{j=1}^{l}$ be the connected components of $M_\epsilon\setminus M_{H_{\epsilon,3m}}$. 

Let $H_{\epsilon,3m}\times [-1,2]$ be the product manifold and we construct a metric $H_{\epsilon,3m}\times [-1,2]$ as in \cite{SWZ}.

Assume $f(Y_j)=p_j \in S^{n}(1)$, $j= 1,\dots, l$. We choose a point $p_0\in S^{n}(1)$ and for $j= 1,\dots,l$, pick a curve $\xi_j(\tau), 0\leq \tau\leq 1$, connecting $p_j$ and $p_0$ such that $\xi_i(\tau)\cap \xi_j(\tau')=\emptyset, 0<\tau, \tau'<1, i\neq j$. Following \cite{SWZ}, for $(y,\tau)\in H_{\epsilon,3m}\times [-1,2]$ and $j=1,\dots,l$, we define\footnote{A similar trick also appears in \cite{CZ}.}
\begin{align}\label{r3.4}
f_\epsilon(y,\tau)=
\begin{cases}
p_j,\ (y,\tau)\in (Y_j\cap H_{\epsilon,3m})\times [-1,0],\\
\xi_j(\tau),\ (y,\tau)\in (Y_j\cap H_{\epsilon,3m})\times [0,1],\\
p_0,\ (y,\tau)\in (Y_j\cap H_{\epsilon,3m})\times [1,2].
\end{cases}
\end{align}
Note that some points of $\{p_j\}_{j=1}^l$ may coincide. Without loss of generality, we assume that  $(0,\cdots,0,\pm 1)\notin \xi_j(\tau), 1\leq j\leq l$. The map $f_\epsilon$ can be extended to a map on $M_{H_{\epsilon,3m}}\cup (H_{\epsilon,3m}\times [-1,2])$ via $f_\epsilon(y,\tau)$. Denote such a map on $M_{H_{\epsilon,3m}}\cup (H_{\epsilon,3m}\times [-1,2])$ by $f_{\epsilon,l}$.

Let $M'_{H_{\epsilon,3m}}$ be another copy of $M_{H_{\epsilon,3m}}$ with the same metric and the opposite orientation. As in \cite{SWZ}, we can glue $M_{H_{\epsilon,3m}}$, $H_{\epsilon,3m}\times [-1,2]$ and $M'_{H_{\epsilon,3m}}$ together to get a closed manifold $\widehat{M}_{H_{\epsilon,3m}}$.
We view $M_{H_{\epsilon,3m}}$, $M'_{H_{\epsilon,3m}}$ and $H_{\epsilon,3m}\times [-1,2]$ as submanifolds of $\widehat{M}_{H_{\epsilon,3m}}$. The map $f_{\epsilon,l}$ can be extended to $\widehat{M}_{H_{\epsilon,3m}}$ by setting $f_{\epsilon,l}(M'_{H_{\epsilon,3m}})={p_0}$.
We still denote the map on $\widehat{M}_{H_{\epsilon,3m}}$ by $f_{\epsilon,l}$. The map $f_{\epsilon,l}$ has the following properties: 
\begin{align}\label{6/3}{\rm Supp}({\rm d}f_{\epsilon,l})\subseteq {\rm Supp}({\rm d}f_\epsilon)\cup (H_{\epsilon,3m}\times [0,1]),\ {\rm deg}(f_{\epsilon,l})={\rm deg}(f_\epsilon)\neq 0.
\end{align}

For any $\beta>0$, let $g^{TM_\epsilon}_{\beta}$ be the Riemannian metric on $M_\epsilon$ defined by
\begin{align}\label{1}
g^{TM_\epsilon}_{\beta}=\beta^2g^{TM_\epsilon}.
\end{align}

Let $g^{TH_{\epsilon,3m}}$ be the induced metric on $H_{\epsilon,3m}$ by (\ref{1}) with $\beta=1$ and ${\rm d}t^2$ be the standard metric on $[0,1]$. By the construction of $\widehat{M}_{H_{\epsilon,3m}}$, we can define a smooth metric $g_{\beta}^{T\widehat{M}_{H_{\epsilon,3m}}}$ on $\widehat{M}_{H_{\epsilon,3m}}$ in the following way:
\begin{align}\label{r3.6}
  \left . g_{\beta}^{T\widehat{M}_{H_{\epsilon,3m}}}\right|_{M_{H_{\epsilon,3m}}} = g^{TM_\epsilon}_{\beta}, \ 
   \left . g_{\beta}^{T\widehat{M}_{H_{\epsilon,3m}}}\right|_{M'_{H_{\epsilon,3m}}} = g^{TM_{H_{\epsilon,3m}}'}, \ 
 \left .   g_{\beta}^{T\widehat{M}_{H_{\epsilon,3m}}}\right|_{H_{\epsilon,3m}\times [0,1]} = g^{TH_{\epsilon,3m}}\oplus {{\rm d} t^2},
   \end{align}
and then paste these metrics together. Let $\nabla_\beta^{T\widehat{M}_{H_{\epsilon,3m}}}$ be the Levi-Civita connection on $T\widehat{M}_{H_{\epsilon,3m}}$ associated with the metric $g_{\beta}^{T\widehat{M}_{H_{\epsilon,3m}}}$.

\subsection{The even-dimensional case}

In this subsection, we assume $n$ is even.
Let $S_{\beta}(T\widehat{M}_{H_{\epsilon,3m}})=S_{\beta,+}(T\widehat{M}_{H_{\epsilon,3m}})\oplus S_{\beta,-}(T\widehat{M}_{H_{\epsilon,3m}})$ be the $\mathbb{Z}_2$-graded Hermitian vector bundle of spinors associated with $(T\widehat{M}_{H_{\epsilon,3m}},g_{\beta}^{T\widehat{M}_{H_{\epsilon,3m}}})$, carrying the canonical induced Hermitian connection $\nabla^{S_{\beta}(T\widehat{M}_{H_{\epsilon,3m}})}=\nabla^{S_{\beta,+}(T\widehat{M}_{H_{\epsilon,3m}})}\oplus \nabla^{S_{\beta,-}(T\widehat{M}_{H_{\epsilon,3m}})}$.

Let $S(TS^n(1))=S_{+}(TS^n(1))\oplus S_{-}(TS^n(1))$ be the spinor bundle of $S^n(1)$. Following~\cite[(2.6)]{Z20}, we construct a suitable bundle endomorphism $V$ of $S(TS^n(1))$.
More precisely, by taking any regular value $\mathfrak{q}\in S^n(1)\setminus f_\epsilon(M_\epsilon\setminus K_\epsilon)$ of $f_\epsilon$, we choose $X$ to be a smooth vector field on $S^n(1)$ such that $|X| > 0$ on $S^n(1)\setminus \{\mathfrak{q}\}$.
Let
\begin{equation*}
  v= c(X): S_+(TS^n(1))\to S_-(TS^n(1))
\end{equation*}
be the Clifford action of $X$ and
\begin{equation*}
  v^*: S_-(TS^n(1))\to S_+(TS^n(1))
\end{equation*}
be the adjoint of $v$ with respect to the Hermitian metric on $S_{\pm}(TS^n(1))$.
We define $V$ to be the self-adjoint odd endomorphism $$V=v+v^*:S(TS^n(1))\to S(TS^n(1)).$$
Then there exists $\delta>0$ such that
\begin{align}\label{t2.1}
(f^*_{\epsilon,l}V)^2\geq \delta\ \ {\rm on}\ \ \widehat{M}_{H_{\epsilon,3m}}\setminus {\rm Supp}({\rm d}f_{\epsilon}).
\end{align}

Let 
\begin{equation}
  \left(E_{3m,\pm},g^{E_{3m,\pm}},\nabla^{E_{3m,\pm}}\right)= f_{\epsilon,l}^*\left (S_{\pm}(TS^n(1)),g^{S_{\pm}(TS^n(1))},\nabla^{S_{\pm}(TS^n(1))}\right)
\end{equation}
be the induced Hermitian vector bundle with the Hermitian connection on $\widehat M_{H_{\epsilon,3m}}$. Then $E_{3m}=E_{3m,+}\oplus E_{3m,-}$ is a ${\mathbb{Z}}_2$-graded Hermitian vector bundle over $\widehat M_{H_{\epsilon,3m}}$.

Let $\nabla^{S_{\beta}(T\widehat{M}_{H_{\epsilon,3m}})\widehat{\otimes} E_{3m}}$ be the connection
on $S_{\beta}(T\widehat{M}_{H_{\epsilon,3m}})\widehat{\otimes} E_{3m}$
 induced by $\nabla^{S_{\beta}(T\widehat{M}_{H_{\epsilon,3m}})}$ and $\nabla^{E_{3m,\pm}}$.

 Let $D^{E_{3m}}_{\beta}$ acting on $S_{\beta}(T\widehat{M}_{H_{\epsilon,3m}})\widehat{\otimes} E_{3m}$ be the twisted Dirac operator defined by
\begin{align}\label{y3.6}
D^{E_{3m}}_{\beta}=\sum_{i=1}^{n}c_{\beta}(h_i)\nabla_{h_i}^{S_{\beta}(T\widehat{M}_{H_{\epsilon,3m}})\widehat{\otimes} E_{3m}},
\end{align}
where $\{h_1,\cdots,h_{n}\}$ is a local oriented orthonormal basis of $(T\widehat{M}_{H_{\epsilon,3m}},g_{\beta}^{T\widehat{M}_{H_{\epsilon,3m}}})$, and $c_\beta(\cdot)$ means that the Clifford action is with respect to the metric $g^{T\widehat{M}_{H_{\epsilon,3m}}}_\beta$.

For $\varepsilon>0$, we introduce the following deformation of $D^{E_{3m}}_{\beta}$ on $\widehat M_{H_{\epsilon,3m}}$,
\begin{align}
D^{E_{3m}}_{\beta}+{{\varepsilon f^*_{\epsilon,l}V}\over{\beta}},
\end{align}
and let 
\begin{multline}
\left(D^{E_{3m}}_{\beta}+{{\varepsilon f^*_{\epsilon,l}V}\over{\beta}}\right)_{+}:\Gamma\left(S_{\beta,+}\left(T\widehat{M}_{H_{\epsilon,3m}}\right){\otimes} E_{3m,+}\oplus S_{\beta,-}\left(T\widehat{M}_{H_{\epsilon,3m}}\right){\otimes} E_{3m,-}\right)\\
\to \Gamma\left(S_{\beta,-}\left(T\widehat{M}_{H_{\epsilon,3m}}\right){\otimes} E_{3m,+}\oplus S_{\beta,+}\left(T\widehat{M}_{H_{\epsilon,3m}}\right){\otimes} E_{3m,-}\right)
\end{multline}
be the natural restriction.

By the Lichnerowicz formula, we have
\begin{multline}\label{t2.6}
\left(D^{E_{3m}}_{\beta}+{{\varepsilon f^*_{\epsilon,l}V}\over{\beta}}\right)^2=\left(D^{E_{3m}}_{\beta}\right)^2+\left[D^{E_{3m}}_{\beta},{{\varepsilon f^*_{\epsilon,l}V}\over{\beta}}\right]+\left({{\varepsilon f^*_{\epsilon,l}V}\over{\beta}}\right)^2\\
=-\Delta^{E_{3m},\beta}+{{k^{T\widehat{M}_{H_{\epsilon,3m}}}}\over{4}}+{1\over 2}\sum_{i,j=1}^{n}R^{E_{3m}}(h_i,h_j)c_\beta(h_i)c_\beta(h_j)+\left[D^{E_{3m}}_{\beta},{{\varepsilon f^*_{\epsilon,l}V}\over{\beta}}\right]+\left({{\varepsilon f^*_{\epsilon,l}V}\over{\beta}}\right)^2,
\end{multline}
where $-\Delta^{E_{3m},\beta}\geq 0$ is the corresponding Bochner Laplacian, $k^{T\widehat{M}_{H_{\epsilon,3m}}}$ is the scalar curvature of $g_{\beta}^{T\widehat{M}_{H_{\epsilon,3m}}}$ and
$$R^{E_{3m}}=\left(\nabla^{E_{3m,+}}\right)^2+\left(\nabla^{E_{3m,-}}\right)^2.$$

On $TM_{H_{\epsilon,3m}}$, by the $(\epsilon,\Lambda^2)$-contracting property of $f_\epsilon$ for the metric $k^{TM_\epsilon}g^{TM_\epsilon}$, we have
\begin{align}\label{tt2.7}
|f_{\epsilon,l,*}(h_i\wedge h_j)|={1\over \beta^2}|f_{\epsilon,l,*}(\beta h_i\wedge \beta h_j)|\leq {\epsilon\over {\beta^2}}|\beta h_i\wedge \beta h_j|_{k^{TM_\epsilon}g^{TM_\epsilon}}={{\epsilon k^{TM_\epsilon}}\over{\beta^2}}, \ \ i\neq j.
\end{align}

Let $\nabla^{S(TS^{n}(1))}$ be the canonical connection on the spinor bundle of $S^{n}(1)$. Let $R^{S(TS^{n}(1))}$ be the curvature tensor of the connection. Set
\begin{align}\label{cc4.5}
C_1=\sup_{p\in S^{n}(1)}\Big|R_p^{S\left(TS^{n}(1)\right)}\Big|.
\end{align}

For $x\in {\rm Supp}({\rm d}f_\epsilon)$ and $s\in\Gamma(\widehat{M}_{H_{\epsilon,3m}}, S_{\beta}(T\widehat{M}_{H_{\epsilon,3m}})\widehat{\otimes} E_{3m})$, by (\ref{tt2.7}), we have
\begin{multline}\label{t2.8}
\left|\left({1\over 2}\sum_{i,j=1}^{n}R^{E_{3m}}(h_i,h_j)c_\beta(h_i)c_\beta(h_j)s,s\right)(x)\right|\\
=\left|\left({1\over 2}\sum_{i,j}^{n}f^*_{\epsilon,l}(R^S(f_{\epsilon,l,*}(h_i\wedge h_j)))c_\beta(h_i)c_\beta(h_j)s,s\right)(x)\right|
\leq {{\epsilon k^{TM_\epsilon}}\over{2\beta^2}} n(n-1)C_1|s|^2(x),
\end{multline}
where $R^S$ is the shorthand for $R^{S(TS^n(1))}$.

Now, we choose 
\begin{align}\label{n2.10}
\epsilon={1\over{4C_1(n+1)^2}}.
\end{align}
Then $f_\epsilon$ is fixed and ${\rm Supp}({\rm d}f_\epsilon)$ is a fixed compact set. Hence, we can find $\kappa>0$ such that
\begin{align}\label{t2.10}
k^{TM_\epsilon}\geq \kappa\ \ {\rm on}\ \ {\rm Supp}({\rm d}f_\epsilon).
\end{align}

On $\widehat{M}_{H_{\epsilon,3m}}\setminus (({\rm Supp}({\rm d}f_\epsilon))\cup (H_{\epsilon,3m}\times [0,1]))$, we have
\begin{align}\label{nn2.12}
\left[D^{E_{3m}}_{\beta},{{\varepsilon f^*_{\epsilon,l}V}\over{\beta}}\right]=0,
\end{align}
and on $H_{\epsilon,3m}\times [0,1]$, we have
\begin{align}\label{nn2.13}
\left[D^{E_{3m}}_{\beta},{{\varepsilon f^*_{\epsilon,l}V}\over{\beta}}\right]=O_{m,\epsilon}\left({\varepsilon\over \beta}\right).
\end{align}

On ${{\rm Supp}({\rm d}f_\epsilon)}$, we have
\begin{align}\label{y2.14}
\left[D^{E_{3m}}_{\beta},{{\varepsilon f^*_{\epsilon,l}V}\over{\beta}}\right]=O_{\epsilon}\left({\varepsilon\over {\beta^2}}\right).
\end{align}

Following \cite[Theorem 1.17]{GL83}, let $\phi:[0, \infty) \rightarrow [0,1]$ be a smooth function such that $\phi \equiv 1$ on $[0, 1]$, $\phi \equiv 0$ on $[2, \infty)$ and $\phi' \approx -1$ on $[1, 2]$. We define a smooth function {${\psi_m}:M_{H_{\epsilon,3m}}\to [0,1]$ by
  \begin{equation}\label{6.6}
    \psi_m(x) = \phi\left(\frac{d(x)}{m}\right),
  \end{equation}
  where $m\in \mathbb{N}$. We extend $\psi_m$ to $(H_{\epsilon,3m}\times [-1,2])\cup M'_{H_{\epsilon,3m}}$ by setting
  \begin{equation*}
    \psi_m\big((H_{\epsilon,3m}\times [-1,2])\cup M'_{H_{\epsilon,3m}}\big)=0.
  \end{equation*}

Following~\cite[p. 115]{BL91}, let $\psi_{m,1},\, \psi_{m,2}: \widehat{M}_{H_{\epsilon,3m}}\rightarrow [0,1]$ be defined by
\begin{align}\label{0.20}
  \psi_{m,1} =\frac{\psi_{m}}{\big(\psi_{m}^2+(1-\psi_{m})^2\big)^{\frac{1}{2}}},\; \psi_{m,2} =\frac{1-\psi_{m}}{\big(\psi_{m}^2+(1-\psi_{m})^2\big)^{\frac{1}{2}}}.
\end{align}
Using the above definition and (\ref{eq:fun-d}), for $i=1,2$, we have
\begin{equation}
  \label{eq:est-psi}
  |\nabla \psi_{m,i}|(x)\leq {C/m} \text{ for any }x\in \widehat{M}_{H_{\epsilon,3m}},
\end{equation}
where $C$ is a constant independent of $g^{TM_\epsilon}$.

For any $s\in \Gamma(\widehat{M}_{H_{\epsilon,3m}},S_{\beta}(T\widehat{M}_{H_{\epsilon,3m}})\widehat{\otimes} E_{3m})$, by (\ref{0.20}), one has
\begin{multline}
\left\|\left(D^{E_{3m}}_{\beta}+{{\varepsilon f^*_{\epsilon,l}V}\over{\beta}}\right)s\right\|^2_\beta\\
=\left\|\psi_{m,1}\left(D^{E_{3m}}_{\beta}+{{\varepsilon f^*_{\epsilon,l}V}\over{\beta}}\right)s\right\|^2_\beta+\left\|\psi_{m,2}\left(D^{E_{3m}}_{\beta}+{{\varepsilon f^*_{\epsilon,l}V}\over{\beta}}\right)s\right\|^2_\beta,
\end{multline}
from which one gets
\begin{multline}\label{yy3.25}
\sqrt{2}\left\|\left(D^{E_{3m}}_{\beta}+{{\varepsilon f^*_{\epsilon,l}V}\over{\beta}}\right)s\right\|_\beta\\
\geq \left\|\psi_{m,1}\left(D^{E_{3m}}_{\beta}+{{\varepsilon f^*_{\epsilon,l}V}\over{\beta}}\right)s\right\|_\beta+\left\|\psi_{m,2}\left(D^{E_{3m}}_{\beta}+{{\varepsilon f^*_{\epsilon,l}V}\over{\beta}}\right)s\right\|_\beta\\
\geq \left\|\left(D^{E_{3m}}_{\beta}+{{\varepsilon f^*_{\epsilon,l}V}\over{\beta}}\right)(\psi_{m,1}s)\right\|_\beta+\left\|\left(D^{E_{3m}}_{\beta}+{{\varepsilon f^*_{\epsilon,l}V}\over{\beta}}\right)(\psi_{m,2}s)\right\|_\beta\\-\|c_\beta({\rm d}\psi_{m,1}s)\|_\beta-\|c_\beta({\rm d}\psi_{m,2}s)\|_\beta,
\end{multline}
where for each $i\in \{1,2\}$, we identify ${\rm d}\psi_{m,i}$ with the gradient of $\psi_{m,i}$.

For any $x\in M_{H_{\epsilon,3m}}$ and $i\in \{1,2\}$, as in \cite[(2.19)]{SWZ}, we have
\begin{align}
|c_\beta({\rm d}\psi_{m,i})s|_{\beta}(x)=O\left({1\over {\beta m}}\right)|s|_{\beta}(x).
\end{align}

Now we estimate the first term on the right-hand side of (\ref{yy3.25}). Using (\ref{t2.1}), (\ref{t2.6}), (\ref{t2.8})-(\ref{t2.10}) and (\ref{y2.14}), we have
\begin{multline}
\left\|\left(D^{E_{3m}}_{\beta}+{{\varepsilon f^*_{\epsilon,l}V}\over{\beta}}\right)(\psi_{m,1}s)\right\|^2_\beta\\ \geq  \left({{k^{TM_\epsilon}}\over{4\beta^2}}\psi_{m,1}s,\psi_{m,1}s\right)_\beta+\left({1\over 2}\sum_{i,j=1}^{n}R^{E_{3m}}(h_i,h_j)c_\beta(h_i)c_\beta(h_j)\psi_{m,1}s,\psi_{m,1}s\right)_\beta\\
+\left(\left[D^{E_{3m}}_{\beta},{{\varepsilon f^*_{\epsilon,l}V}\over{\beta}}\right]\psi_{m,1}s,\psi_{m,1}s\right)_\beta+\left(\left({{\varepsilon f^*_{\epsilon,l}V}\over{\beta}}\right)^2\psi_{m,1}s,\psi_{m,1}s\right)_\beta
\\
\geq {\kappa\over{8\beta^2}} \|s\|^2_{\beta,{\rm Supp}({\rm d}f_\epsilon)}+O_{\epsilon}\left({\varepsilon\over{\beta^2}}\right)\|s\|^2_{\beta,{\rm Supp}({\rm d}f_\epsilon)}+{{\delta\varepsilon^2}\over{\beta^2}}\|\psi_{m,1}s\|^2_{\beta,B_{2m}\setminus {\rm Supp}({\rm d}f_\epsilon)}.
\end{multline}

For the second term on the right-hand side of (\ref{yy3.25}), using (\ref{t2.1}) and (\ref{nn2.13}), we have
\begin{align}\label{yyy3.29}
\left\|\left(D^{E_{3m}}_{\beta}+{{\varepsilon f^*_{\epsilon,l}V}\over{\beta}}\right)(\psi_{m,2}s)\right\|^2_\beta
\geq {{\delta\varepsilon^2}\over{\beta^2}}\|\psi_{m,2}s\|^2_\beta+O_{m,\epsilon}\left({\varepsilon\over\beta}\right)\|\psi_{m,2}s\|^2_{\beta,H_{\epsilon,3m}\times [0,1]}.
\end{align}

Then using (\ref{yy3.25})-(\ref{yyy3.29}), one finds that there exist $c_0>0,\varepsilon>0,m>0$ such that when $\beta>0$ is small enough, 
\begin{align}\label{yy3.29}
\left\|\left(D^{E_{3m}}_{\beta}+{{\varepsilon f^*_{\epsilon,l}V}\over{\beta}}\right)s\right\|_\beta\geq {c_0\over\beta}\|s\|_\beta.
\end{align}

On the other hand, by the Atiyah-Singer index theorem \cite{AS} (cf. \cite[Proposition III. 13.8]{LaMi89}), as in \cite[(2.44)]{SWZ}, we have
\begin{multline}
{\rm ind}\left(\left(D^{E_{3m}}_{\beta}+{{\varepsilon f^*_{\epsilon,l}V}\over{\beta}}\right)_{+}\right)\\
= \left\langle \widehat{{A}}(TM_\epsilon) f_{\epsilon,l}^*\left({\rm ch}(S_+(TS^n(1)))-{\rm ch}(S_-(TS^n(1)))\right), [M_\epsilon] \right\rangle \\
    ={\rm deg}(f_{\epsilon,l}) \Big\langle {\rm ch}(S_+(TS^n(1)))-{\rm ch}(S_-(TS^n(1))),[S^n(1)]\Big\rangle \\
    = (-1)^{n\over 2}{\rm deg}(f_{\epsilon})\chi(S^n(1)) =2(-1)^{n\over 2}{\rm deg}(f_{\epsilon}) \neq 0.\\
\end{multline}
Then we get a contradiction. 

\subsection{The odd-dimensional case}
In this subsection, we assume $n$ is odd. We consider the composition $f_{\epsilon,l,r}$ of the maps
\begin{align}
\widehat{M}_{H_{\epsilon,3m}}\times S^1(r)\xrightarrow{f_{\epsilon,l}\times {1\over r}{\rm id}}S^n(1)\times S^1(1)\xrightarrow{\wedge}S^{n+1}(1),
\end{align}
where $S^1(r)$ is the round circle of radius $r$ with the canonical metric.

Fix $\epsilon$ as (\ref{n2.10}) and set
$$\kappa_0=\min \left\{k^{TM_\epsilon}(x):x\in {\rm Supp}({\rm d}f_\epsilon)\right\}.$$

We choose $r$ large enough such that 
 \begin{align}
{ {\sup\left\{|{{\rm d}}f_{\epsilon}|(x),x\in M_\epsilon\right\}}\over r}<\epsilon \kappa_0.
 \end{align}
Then by combining the method used in the above even-dimensional case and \cite[Section 3]{SWZ}, we can also get a contradiction.

One can also give a direct proof for odd-dimensional case using the argument in \cite{LSWZ}. We briefly outline it here.

Let $E$ be the $\mathbb{Z}_2$-graded vector bundle defined in \cite[p. 3]{LSWZ}. Then as (\ref{tt2.7})-(\ref{n2.10}), using \cite[Proposition 4.1]{LSWZ}, we choose and fix $\epsilon$. For simplicity we use the same notations as in \cite{LSWZ}.

Let $D^{\mathcal{E}_{3m},u}_\beta$, $0\leq u\leq 1$, be the family of twisted Dirac operators defined in \cite[(3.10)]{LSWZ}. Let $V:\Gamma(S_{\beta}(T\widehat{M}_{H_{3m}}){\otimes}\mathcal{E}_{3m})\to \Gamma(S_{\beta}(T\widehat{M}_{H_{3m}}){\otimes}\mathcal{E}_{3m})$ be the operator defined in \cite[p. 7]{LSWZ}. For any $\varepsilon>0$, let $D^{\mathcal{E}_{3m},u}_{\beta,\varepsilon}:\Gamma(S_{\beta}(T\widehat{M}_{H_{3m}}){\otimes}\mathcal{E}_{3m})\to \Gamma(S_{\beta}(T\widehat{M}_{H_{3m}}){\otimes}\mathcal{E}_{3m})$, $0\leq u\leq 1$, be the family of deformed twisted Dirac operators defined by
\begin{align}\label{2.5}
D^{\mathcal{E}_{3m},u}_{\beta,\varepsilon}=D^{\mathcal{E}_{3m},u}_{\beta}+{{\varepsilon V}\over{\beta}}.
\end{align}
Note that we do not need the cut-off function $\varphi$ (\cite[p. 7]{LSWZ}) here. Then proceeding as (\ref{t2.10})-(\ref{yyy3.29}), one finds that there exist $\varepsilon>0$, $m\in \mathbb{N}$ and $\beta>0$ such that the operator $D^{\mathcal{E}_{3m},u}_{\beta,\varepsilon}$ is invertible for any $u\in [0,1]$. On the other hand, by the same proof in \cite[Section 5]{LSWZ}, we can get ${\rm deg}(f_\epsilon)=0$, which contradicts the assumption that ${\rm deg}(f_\epsilon)\neq 0$.

$\ $

\noindent{\bf Acknowledgments.}
The author would like to thank Pengshuai Shi, Changliang Wang, Xiangsheng Wang and Prof. Weiping Zhang for helpful discussions. The author also would like to thank the referees for their helpful suggestions and comments. This work was partially supported by National Natural Science Foundation of China (Grant No. 12425106, Grant No. 12271266), the Fundamental Research Funds for the Central Universities 63253096, Natural Science Foundation of Tianjin, China (Grant No. 25JCZDJC01110) and Nankai Zhide Foundation.

\end{document}